\documentclass[UTF8,a4paper,10pt]{elsarticle}
\usepackage[left=2.50cm, right=2.50cm, top=2.50cm, bottom=2.50cm]{geometry} 

\usepackage{latexsym}
\usepackage{amssymb}
\usepackage{newlfont}
\usepackage{hyperref}
\usepackage{mathrsfs}
\usepackage{amsmath}
\usepackage{amsthm}
\usepackage{booktabs}
\usepackage{amsfonts}
\usepackage{mathrsfs}

\newtheorem{thm}{Theorem}[section]

\newtheorem{prop}[thm]{Proposition}
\newtheorem{Def}[thm]{Definition}
\newtheorem{lemma}[thm]{Lemma}
\newtheorem{remark}[thm]{Remark}
\numberwithin{equation}{section}
\newcommand{\n}{||}
\newcommand{\e}{\mathcal{E}}
\newcommand{\hh}{H^2\cap H^1_0}
\newcommand{\un}{\tilde{u}}
\newcommand{\um}{\bar{u}}

\newcommand{\fot}{G_{u,v}}
\newcommand{\bw}{\tilde{B}}

\begin{document}

\begin{frontmatter}
\title{\textbf{Strong attractors for weakly damped quintic wave equation in bounded domains}}
\author{ Senlin Yan*, Zhijun Tang, Chengkui Zhong }

\address{Department of Mathematics, Nanjing University, Nanjing, 210093, PR China}

\cortext[]{Corresponding author.\\
	E-mails: dg20210019@smail.nju.edu.cn(S. Yan), tzj960629@163.com(Z. Tang), ckzhong@nju.edu.cn(C. Zhong), }

\begin{abstract}
	In this paper, we study the longtime dynamics for the weakly damped wave equation with quintic non-linearity in a bounded smooth domain of $\mathbb{R}^3.$ Based on the Strichartz estimates for the case of bounded domains, we establish the existence of a strong global attractor in the phase space $H^2(\Omega)\cap H^1_0(\Omega)\times H^1_0(\Omega)$. Moreover, the finite fractal dimension of the attractor is also shown with the help of the quasi-stable estimation.
\end{abstract}

	
\end{frontmatter}

\section{Introduction}
	In this paper, we consider the following damped wave equation:
	\begin{equation}\label{eq1.1}
		\begin{cases}
			u_{tt}+\gamma u_t-\Delta u+f(u)=g,\quad (x,t)\in\Omega\times\mathbb{R}^+,\\
			u|_{\partial\Omega}=0,\quad t\in\mathbb{R}^+,\\
			u|_{t=0}=u_0,\ u_t|_{t=0}=u_1,\quad x\in\Omega.
		\end{cases}
	\end{equation}
	Here, $\Omega\subset\mathbb{R}^3$ is a bounded domain with smooth boundary, $\gamma>0$ is a fixed constant, $g\in L^2(\Omega)$ is independent of time, and the non-linearity $f\in C^2(\mathbb{R})$ with $f(0)=0$ and satisfies the following growth and dissipation conditions:
	\begin{align}
		\label{1.2}&|f''(s)|\leq C(1+|u|^3),\quad s\in\mathbb{R},\\
		\label{1.3}&\liminf_{|s|\rightarrow\infty}\frac{f(s)}{s}>-\lambda_1,\\
		\label{1.4}&f'(s)\geq -K,
	\end{align}
    where the constants $C,K\geq0$ and $\lambda_1>0$ is the first eigenvalue of $-\Delta$ on $L^2(\Omega)$ with Dirichlet boundary conditions. Additionally, in order to deal with the quintic case, we assume that $f$ satisfies the following extra conditions:
    \begin{align}
    	\label{1.5}&F(s)\geq-C+\sigma|s|^6,\quad\sigma>0,\\
    	\label{1.6}&f(s)s-4F(s)\geq-C,\\
    	\label{1.7}&f'(s)\geq-K+\delta|s|^4,\quad\delta>0,
    \end{align}
    where $C>0$ and $F(s)=\int_{0}^{s}f(\tau)d\tau$.\\
    \indent Weakly damped semilinear wave equations model various oscillatory processes in many areas of modern mathematical physics including electrodynamics, quantum mechanics, non-linear elasticity, see e.g. \cite{Arrieta,Babin,Ball,Vishik,Temam}. It is believed that the long-time dynamics of problem (\ref{eq1.1}) depends strongly on the growth rate of the non-linerity $f$. To the best of our knowledge, in a bounded smooth domain of $\mathbb{R}^3$, one can verify the global well-posedness of problem (\ref{eq1.1}) only in the case of sub-cubic or cubic non-linearities. Therefore, for a long time, the cubic growth rate of non-linearity $f$ has been considered as a critical one for the case of 3-D bounded domain and the existence of a compact global attractor in the natural energy space $H^1_0(\Omega)\times L^2(\Omega)$ has been known only for the sub-cubic or cubic case, see \cite{Arrieta,Babin,Ghidaglia,Hale1} and references therein. \\
    \indent For the case of sup-cubic non-linearity, since the uniqueness of weak energy solutions is still an open problem, one has to consider weak trajectory attractors or to consider the solution semigroup as multivalued maps, see \cite{Carvalho,Zelik3} and references therein. However, with the help of suitable versions of Strichartz estimates and Morawetz-Pohozhaev identity in the case of bounded domains, one can obtain the global well-posedness of so-called Shatah-Struwe solutions (Definition \ref{SS}), see \cite{Blair,Burq1,Burq2}. Working on the Shatah-Struwe solution semigroup arising from problem (\ref{eq1.1}), Kalantarov et al. \cite{Zelik1} proved the existence and regularities of the compact global attractor in the case of sub-quintic and quintic growth rates of non-linearitiy $f$. Then, Liu et al. \cite{Liu5,Liu6,Liu4} established the translational regular solution and studied the long-time behaviour of problem (\ref{eq1.1}) with lower regular forcing ($g\in H^{-1}$) and super-cubic non-linearity in both bounded and unbounded domains in $\mathbb{R}^3$. Recently, Savostianov and Zelik \cite{Savostianov1} studied the damped quintic wave equations with measure-driven and non-autonomous external forces in the 3-D perodic boundary conditions case. They proved the exitence of uniform attractors in a weak or strong topology in the energy phase space and their additional regularities. After that, Mei et al. \cite{Mei} studied the infinite-energy solutions for weakly damped wave equations in the whole space $\mathbb{R}^3$, which can be considered as a continuation of \cite{Savostianov1}.\\
    \indent Concerning the strong solution, Babin and Vishik \cite{Babin2} proved the existence of the maximal $(\e_1,\e)$ (for the definitions of the  symbols see (\ref{2-0})) atractor for problem (\ref{eq1.1}) by requesting the initial data belongs to $\e_1$. Ladyzhenskaya \cite{Lady} established the $\e_1$ attractor for cubic non-linearity. In \cite{Meng}, the author verified the $\e_1$ attractor via a different approach. Resently, the authors in \cite{Liu1} have obtained the strong global attractor in the sup-cubic but sub-quintic non-linearity case and they also obtained the strong exponential attractor, see \cite{Liu3}.\\
    \indent In this paper, we are concerned with the strong solutions of problem (\ref{eq1.1}) in the quintic case. We note that in \cite{Zelik1}, the authors have proved that if the initial data is taken form more regular phase space, then the corresponding Shatah-Struwe solution is more regular too. Thus it is natural to consider the strong attractor of Shatah-Struwe solution semigroup when restricted on $\e_1$. The main difficulty is to prove the dissipativity of the corresponding dynamical system. In \cite{Liu1}, when considering the sub-quintic case, the authors concerned the point dissipative since it seemed hard to get uniform dissipativity. However, we will show that it is actually possible to obtain the uniform dissipativity both in the sub-quintic and in the quintic case. Then, we use the quasi-stable method established in \cite{Chueshov2015} to prove that the strong global attractor exists and it has a finite fractal dimension.\\
    \indent The rest of the paper is organized as folows. In Section 2, some preliminary things including Strichartz estimates and the quasi-stable method are disscussed. In Section 3, we give some difinitions and prove the well-posedness of strong solutions. Uniform dissipativity of the strong solution semigroup is discussed in Section 4 and the quasi-stability of the dynamical system is obtained in Section 5. Finally, the existence and finite dimensionality of the strong global attractor are proved in Section 6.

\section{Preliminaries}
	In this section, we recall some concepts and mathematical preliminaries that will be used later.\\
	\indent We begin with the following abbreviations:
	\[L^p=L^p(\Omega),\quad H^1_0=H^1_0(\Omega),\quad H^2=H^2(\Omega),\]
	with $p\geq1.$ The notation $(\cdot,\cdot)$ stands for the usual inner product in $L^2.$ $X\hookrightarrow Y$ denotes that the space $X$ continuously embeds into $Y$ and $X\hookrightarrow\hookrightarrow Y$ denotes that $X$ compactly embeds into $Y.$ Define the following phase spaces
	\begin{equation}\label{2-0}
		\e=H^1_0\times L^2,\quad \e_1=[\hh]\times H^1_0,
	\end{equation}
	equipped with the norms respectively
	\[\n\xi_u\n_\e^2=\n\nabla u\n_{L^2}^2+\n u_t\n_{L^2}^2,\quad \n\xi_u\n_{\e_1}^2=\n\Delta u\n_{L^2}^2+\n\nabla u_t\n_{L^2}^2,\]
	where the vector $\xi_u:=(u,u_t)$.
	
\subsection{Strichartz type estimates.}
	We first recall the Strichartz estimates for wave equations in bounded domains. (see \cite{Blair,Burq1,Burq2} for details).
	\begin{equation}\label{eq2.1}
		\begin{cases}
			u_{tt}+\gamma u_t-\Delta u=G(t),\quad (x,t)\in\Omega\times\mathbb{R}^+,\\
			u|_{\partial\Omega}=0,\quad t\in\mathbb{R}^+,\\
			\xi_u(0)=\xi_0,\quad x\in\Omega.
		\end{cases}
	\end{equation}
    \begin{lemma}\citep[Proposition 2.1]{Zelik1}\label{lemma2.1}
    	Let $\xi_0\in\e,\ G\in L^1(0,T;L^2)$ and $u(t)$ be a solution of equation (\ref{eq2.1}) such that $\xi_u\in C(0,T;\e)$. Then the following estimate holds:
    	\begin{equation}
    		\n\xi_u(t)\n_{\e}\leq C\left(\n\xi_0\n_{\e}e^{-\beta t}+\int_{0}^{t}e^{-\beta(t-s)}\n G(s)\n_{L^2}ds,\right)
    	\end{equation}
    where the positive constants $\beta,\ C$ depend on $\gamma$, but independ on $t,\ \xi_0,$ and $G$.
    \end{lemma}
	\begin{lemma}\citep[Proposition 2.2]{Zelik1}\label{lemma2.2}
		Let the assumptions of Lemma \ref{lemma2.1} hold. Then, $u\in L^4(0,T;L^{12})$ and the following estimate holds:
		\begin{equation}\label{Stri}
			\n u\n_{L^4(0,T;L^{12})}\leq C_T(\n\xi_0\n_\e+\n G\n_{L^1(0,T;L^{2})}),
		\end{equation}
	where the constant $C_T$ may depend on $T$, but is independent of $\xi_0$ and $G$.
	\end{lemma}
    \begin{remark}
    	Combining with the above two Lemmas, we get 
    	\begin{equation}\label{esti2.4}
    		\n\xi_u(t)\n_{\e}+\n u\n_{L^4(max\{0,t-1\},t;L^{12})}\leq C\left(\n\xi_0\n_{\e}e^{-\beta t}+\int_{0}^{t}e^{-\beta(t-s)}\n G(s)\n_{L^2}ds,\right),
    	\end{equation}
    where the constants $\beta,\ C$ are independent of $u$ and $t\geq0$.
    \end{remark}
	The next basic fact is a direct result of Corollary 2.4 in \cite{Zelik1}.
	\begin{lemma}\label{coro2.4}
		Let $K\subset\e$ be a compact set. Then, for any $\varepsilon>0$ there is $T=T(\varepsilon,K)>0$ such that
		\[\n u\n_{L^4(0,T;L^{12})}\leq\varepsilon\]
		for all solutions $u$ of equation (\ref{eq2.1}) with initial data $\xi_u(0)\in K.$
	\end{lemma}

\subsection{Quasi-stable system}
	In this subsection, we recall the theory of quasi-stability method developed by I. Chueshov and I. Lasiecka\cite{Chueshov2008,Chueshov2015}. This method will help us verifying the asymptotic smoothness of the solution semigroup and the finite dimensionality of the corresponding global attractor.
	\begin{Def}\citep[Definition 3.4.1]{Chueshov2015}\label{defquasi}
		Let $(X,S(t))$ be a dynamical system in some Banach space $X$. This system is said to be quasi-stable on a set $B\subset X$ if there exist a time $t_*>0$, a Banach space $Z$, a globally Lipschitz mapping $K:B\rightarrow Z$ and a compact seminorm $n_Z(\cdot)$ on the space $Z$, such that
		\[\n S(t_*)y_1-S(t_*)y_2\n_X\leq q\n y_1-y_2\n_X+n_Z(Ky_1-Ky_2)\]
		for every $y_1,y_2\in B$ with $q\in[0,1)$. The space $Z$, the operator $K$, the seminorm $n_Z$ and the time moment $t_*$ may depend on $B$.
	\end{Def}
	\begin{prop}\citep[Proposition 3.4.3]{Chueshov2015}\label{prop1}
		Assume that a dynamical system $(X,S(t))$ is quasi-stable on every positivity invariant bounded set $B$ in $X$. Then $(X,S(t))$ is asymptotically smooth.
	\end{prop}
	\begin{prop}\citep[Theorem 3.4.5]{Chueshov2015}\label{prop2}
		Assume that a dynamical system $(X,S(t))$ possesses a compact global attractor $\mathcal{A}$ and is quasi-stable on $\mathcal{A}$ at some time $t_*>0$. Then, $\mathcal{A}$ has a finite fractal dimension $dim_f^X(\mathcal{A})$ in $X$. Moreover, we have the estimate
		\begin{equation}\label{dim}
			dim_f^X(\mathcal{A})\leq\left[\ln\frac{2}{1+q}\right]^{-1}\cdot \ln m_Z\left(\frac{4L_K}{1-q}\right),
		\end{equation}
	where $L_K>0$ is the Lipschitz constant for $K$ and $m_Z(R)$ is the maximal number of elements $z_i$ in the ball $B_Z(0,R)=\{z\in Z:\n z\n_Z\leq R\}$ possessing the property $n_Z(z_i-z_j)>1$ when $i\neq j$.
	\end{prop}
	
\subsection{Some abstract results}
    \begin{lemma}\cite{Savostianov}\label{Savostianov}
    	Suppose that a function $y(s)\in C([a,b))$ satisfies $y(a)=0,\ y(s)\geq0$ and 
    	\[y(s)\leq C_0y(s)^\sigma+\varepsilon\]
    	for some $\sigma>1,\ 0<C_0<\infty$ and $0<\varepsilon<\frac{1}{2}(\frac{1}{2C_0})^\frac{1}{\sigma-1}$. Then
    	\[y(s)\leq2\varepsilon,\quad \forall s\in[a,b).\]
    \end{lemma}
	\begin{lemma}(Aubin-Dubinskii-Lions Lemma,\cite{Lions})\label{Aubin}
		Assume that $X\subset Y\subset Z$ is a triple of Banach spaces such that $X\hookrightarrow\hookrightarrow Y\hookrightarrow Z$.\\
		(i) Let $W_1=\{u\in L^p(a,b;X)|u_t\in L^q(a,b;Z)\}$ for some $1\leq p\leq\infty$ and $q\geq1$. Here $u_t$ denotes the derivative of $u$ in the distributional sense. Then, $W_1\hookrightarrow\hookrightarrow L^p(a,b;Y).$ If $q>1$, then $W_1\hookrightarrow\hookrightarrow C(a,b;Z).$\\
		(ii) Let $W_2=\{u\in L^\infty(a,b;X)|u_t\in L^r(a,b;Z)\}$ for some $r>1$, then $W_2\hookrightarrow\hookrightarrow C(a,b;Y)$.
	\end{lemma}
	\begin{lemma}\label{Gronwall}
		Let $\Phi:\mathbb{R^+}\rightarrow\mathbb{R^+}$ be an absolutely continuous function satisfying
		\[\frac{d}{dt}\Phi(t)+2\alpha\Phi(t)\leq h(t)\Phi(t)+k,\]
		where $\alpha>0,\ k\geq0$ and $\int_{s}^{t}h(\tau)d\tau \leq \alpha(t-s)+m$, for all $t\geq s\geq0$ and some $m\geq0.$ Then
		\[\Phi(t)\leq\Phi(0)e^me^{-\alpha t}+\frac{ke^m}{\alpha},\ \forall t\geq0.\]
	\end{lemma}

\section{Well-posedness}
	Firstly, we recall some definitions and results about the weak solution and Shatah-Struwe solution of problem (\ref{eq1.1}) in \cite{Zelik1}.
	\begin{Def}
		For any $T>0$, a function $u(t),\ t\in[0,T]$ is said to be a weak solution of problem (\ref{eq1.1}) if $\xi_u\in L^\infty(0,T;\e)$ and Eq. (\ref{eq1.1}) is satisfied in the sense of distribution, i.e.
		\[-\int_{0}^{T}(u_t,\phi_t)dt+\gamma\int_{0}^{T}(u_t,\phi)dt+\int_{0}^{T}(\nabla u,\nabla\phi)dt+\int_{0}^{T}(f(u),\phi)dt=\int_{0}^{T}(g,\phi)dt,\]
		for any $\phi\in C_c^\infty((0,T)\times\Omega)$.
	\end{Def}
	Since the uniqueness of the weak solution is still unknown in the sup-cubic case, we need some extra regularity.
	\begin{Def}\citep[Definition 2.8]{Zelik1}\label{SS}
		A weak solution $u(t),\ t\in[0,T]$ is a  Shatah-Struwe solution of problem (\ref{eq1.1}) if the following additional regularity holds:
		\[u\in L^4(0,T;L^{12}).\]
	\end{Def}
	V. Kalantarov et al.\cite{Zelik1} proved the well-posedness of the Shatah-Struwe solution and the existence of a compact global attractor of the corresponding semigroup, namely
	\begin{lemma}\cite{Zelik1}\label{lemma3.3}
		Let $f$ satisfies assumptions (\ref{1.2})-(\ref{1.7}).\\
		(i) For any $\xi_0\in\e$, there exists a unique Shatah-Struwe solution $u(t),\ t\in\mathbb{R^+}$ and this solution satisfies the following dissipative energy estimate:
		\begin{equation}\label{3-1}
			\n\xi_u(t)
			\n_\e\leq Q(\n\xi_0\n_\e)e^{-\alpha t}+Q(\n g\n_{L^2}),\quad t\in\mathbb{R^+},
		\end{equation}
		and the following Strichartz estimate:
		\begin{equation}\label{3-2}
			\n u\n_{L^4(0,T;L^{12})}\leq Q(\xi_0,T),\quad T\geq0,
		\end{equation}
	    where the monotone function $Q$ and the positive constant $\alpha$ are independent of $u$.\\
		(ii) The Shatah-Struwe solution semigroup $(\e,S(t))$ associated with equation (\ref{eq1.1}) possesses a global attractor $\mathcal{A}_0$ in $\e$ which is a bounded set in $\e_1$.
	\end{lemma}
	\begin{remark}\label{remark3.4}
		In the sub-quintic case, i.e. $|f''(s)|\leq C(1+|u|^{3-\kappa})$ with some $0<\kappa\leq3$, we can obtain the same results by requiring $f$ to satisfy (\ref{1.3}) only. Moreover, we have a better estimate than (\ref{3-1}) and (\ref{3-2}):
		\begin{equation}\label{3-3}
			\n\xi_u(t)
			\n_\e+\n u\n_{L^4(t,t+1;L^{12})}\leq Q(\n\xi_0\n_\e)e^{-\alpha t}+Q(\n g\n_{L^2}),\quad t\in\mathbb{R^+}.
		\end{equation}
	\end{remark}
    See also \cite{Zelik1} for details.\\
	\indent In this paper, we are concerned with the strong solution of problem (\ref{eq1.1}), the definition is as follows:
	\begin{Def}
		For any $T>0$, a function $u(t),\ t\in[0,T]$ is said to be a strong solution of problem (\ref{eq1.1}) if $\xi_u\in L^\infty(0,T;\e_1)$ and Eq. (\ref{eq1.1}) is satisfied in $L^2$ for almost all $t\in[0,T]$.
	\end{Def}
    Before discussing the well-posedness of strong solution, we prove a crucial fact that if the initial datas are taken from a bounded set of $\e_1$ then the corresponding Shatah-Struwe solutions have a uniform $L^4(L^{12})-$norm estimate.
    \begin{prop}\label{fact}
    	Let $B\subset\e_1$ be a bounded set. Then there exist a time $T_0=T_0(B)$ and a constant $C=C(B)$ such that
    	\begin{equation}
    		\n u\n_{L^4(0,T_0;L^{12})}\leq C
    	\end{equation}
        for any Shatah-Struwe solution $u(t)$ with initial data $\xi_u(0)\in B$, where the constant $C$ is independent of $u$.
    \end{prop}
    \begin{proof}
    	The proof is similar to the one of Proposition 3.1 in \cite{Zelik1}. We split the solution $u=v+w$ where $v,w$ solve the following problems respectively
    	\[v_{tt}+\gamma v_t-\Delta v=g,\quad \xi_v(0)=\xi_u(0);\]
    	\[w_{tt}+\gamma w_t-\Delta w=-f(v+w),\quad \xi_w(0)=0.\]
    	Due to the compact embedding $\e_1\hookrightarrow\hookrightarrow\e$, the set $[B]_\e$(the closure of $B$ in $\e$) is compact in $\e$. Then by Lemma \ref{coro2.4}, for any $\varepsilon>0$, there exists $T_\varepsilon=T_\varepsilon(\varepsilon,B)>0$, such that
    	\begin{equation}\label{v}
    		\n \xi_v(t)\n_\e\leq C,\quad \n v\n_{L^4(0,t;L^{12})}\leq\varepsilon,\quad \forall t\leq T_\varepsilon,\ \forall\xi_v(0)\in B,
    	\end{equation}
        where $C$ depends on $B$ and $\n g\n_{L^2}$. Using the growth condition (\ref{1.2}) and interpolation inequality, we have $
        \forall t\leq T_\varepsilon,$
        \begin{align*}
        	\n f(v+w)\n_{L^1(0,t;L^{2})}&\leq C\n 1+|v|^5+|w|^5\n_{L^1(0,t;L^{2})}\\
        	&\leq C(t+\n v\n^5_{L^5(0,t;L^{10})}+\n w\n^5_{L^5(0,t;L^{10})})\\
        	&\leq C(t+\n v\n^4_{L^4(0,t;L^{12})}\n v\n_{L^\infty(0,t;H^1_0)}+\n w\n^5_{L^5(0,t;L^{10})})\\
        	&\leq C(t+\varepsilon^4+\n w\n^5_{L^5(0,t;L^{10})}).
        \end{align*}
        Applying estimate (\ref{esti2.4}) to the variable $w$,
        \begin{align*}
        	\n\xi_w(t)\n_\e+\n w\n_{L^5(0,t;L^{10})}+\n w\n_{L^4(0,t;L^{12})}&\leq C\n f(v+w)\n_{L^1(0,t;L^{2})}\\
        	&\leq C(t+\varepsilon^4)+C(\n\xi_w(t)\n_\e+\n w\n_{L^5(0,t;L^{10})})^5.
        \end{align*}
        Denote $Y(t):=\n\xi_w(t)\n_\e+\n w\n_{L^5(0,t;L^{10})}$, then we get
        \[y(t)\leq C(t+\varepsilon^4)+Cy(t)^5,\quad t\leq T_\varepsilon,\quad y(0)=0.\]
        Applying Lemma \ref{Savostianov}, choosing $\varepsilon$ and $T_\varepsilon$ such that $0<C(T_\varepsilon+\varepsilon^4)\leq \frac{1}{2}(\frac{1}{2C})^\frac{1}{4},$ we get
        \[y(t)\leq 2C(T_\varepsilon+\varepsilon^4),\quad\forall t\leq T_\varepsilon.\]
        Now, fixing $\varepsilon$ and $T_\varepsilon=:T_0$ small to satisfy the above inequality and then
        \[\n\xi_w(t)\n_\e+\n w\n_{L^4(0,t;L^{12})}\leq C,\quad \forall t\leq T_0.\]
        Combining the above estimate with (\ref{v}), we end up with
        \[\n u\n_{L^4(0,T_0;L^{12})}\leq C,\quad \forall \xi_u(0)\in B,\]
        where the time $T_0$ and constant $C$ may depend on $B$, but are independent of $u$.

    \end{proof}
	\begin{thm}\label{well}
		Let $T>0$ be arbitrary and $f$ satisfy assumptions (\ref{1.2})-(\ref{1.7}). Then for any initial data $\xi_0\in\e_1$, problem (\ref{eq1.1}) admits a strong solution $u$. Moreover, the strong solutions depend continuously on the initial data in $\e_1$. That is to say, if ${u_1},\ {u_2}$ are two strong solutions corresponding to the initial data $\xi_{u_1}(0),\ \xi_{u_2}(0)$ respectively, then
		\begin{equation}\label{3-6-0}
			\n\xi_{u_1}(t)-\xi_{u_2}(t)\n_{\e_1}\leq C_{r,T}\n\xi_{u_1}(0)-\xi_{u_2}(0)\n_{\e_1},\quad \forall t\in[0,T],
		\end{equation}
		where $r>0$ is a constant such that $\xi_{u_1}(0),\ \xi_{u_2}(0)\in\{\xi\in\e_1:\n\xi\n_{\e_1}\leq r\}.$ In particular, the strong solution is unique.
	\end{thm}
	\begin{proof}
		\textbf{Step 1.} \emph{Existence.} The existence part follows directly from Proposition 3.10 in \cite{Zelik1}, and we write here for completeness. We just need to verify that the Shatah-Struwe solution given in Lemma \ref{lemma3.3} is actually a strong solution while the initial data $\xi_0\in\e_1$ and we give below only the formal proof which can be justified by using Faedo-Galerkin method.\\
		\indent Let $v(t):=u_t(t)$. Then from Eq. (\ref{eq1.1}) and condition (\ref{1.2}),
		\begin{equation}\label{3-6-1}
			\xi_v(0)=(u_t(0),u_{tt}(0))=(u_1,\Delta u_0-f(u_0)-\gamma u_1+g)\in\e
		\end{equation}
		and the function $v$ solves
		\begin{equation}\label{eq3.1}
		v_{tt}+\gamma v_t-\Delta v=-f'(u)v,\quad\xi_v(0)\in\e.
		\end{equation}
	    Multiplying Eq. (\ref{eq3.1}) by $v_t$ and integrating over $x\in\Omega$, we get
	    \begin{equation}\label{3-6-2}
	    	\frac{1}{2}\frac{d}{dt}\n\xi_v(t)
	    	\n_\e^2+\gamma\n v_t\n_{L^2}^2=-(f'(u)v,v_t).
	    \end{equation}
	    The right hand side of (\ref{3-6-2}) obeys the estimate
	    \begin{align*}
	    	|(f'(u)v,v_t)|&\leq C((1+|u|^4)|v|,|v_t|)\\
	    	&\leq C\n1+|u|^4\n_{L^3}\n v\n_{L^6}\n v_t\n_{L^2}\\
	    	&\leq C(1+\n u\n_{L^{12}}^4)\n \xi_v\n_\e^2.
	    \end{align*}
        Inserting the above estimate into (\ref{3-6-2}) and using Gronwall inequality, we have
        \begin{equation}\label{3-6-3}
        	\n\xi_v(t)\n_\e^2\leq \n\xi_v(0)\n_\e^2\exp\left(CT+\int_{0}^{T}\n u\n_{L^{12}}^4ds\right),\quad \forall t\in[0,T].
        \end{equation}
        Now, we rewrite Eq. (\ref{eq1.1}) in the form
        \[-\Delta u(t)+f(u(t))=g-v_t(t)-\gamma v(t)\in L^2.\]
        Multiplying this elliptic equation by $-\Delta u$ in $L^2$ and using the condition (\ref{1.4}), we obtain
        \begin{equation}\label{3-6-4}
        	\n \Delta u(t)\n_{L^2}^2\leq C(\n\xi_v(t)\n_\e^2+\n g\n_{L^2}^2)+K\n\xi_u(t)\n_\e^2,\quad \forall t\in[0,T].
        \end{equation} 
        Combining with estimates (\ref{3-1}), (\ref{3-6-1}), (\ref{3-6-3}), (\ref{3-6-4}), we end up with
        \begin{equation}\label{3-6-5}
        	\begin{split}
        		\n\xi_u(t)\n_{\e_1}&\leq C(\n\xi_v(t)\n_\e+\n g\n_{L^2}+\n\xi_u(t)\n_\e)\\
        		&\leq Q(\n\xi_0\n_{\e_1})\exp(CT+C\n u\n_{L^4(0,T;L^{12})})+Q(\n g\n_{L^2},\n \xi_0\n_{\e}),\ \forall t\in[0,T],
        	\end{split}
        \end{equation}
        where the monotone function $Q$ and the constant $C$ is independent of $u$. \\
         \indent Moreover, due to Proposition \ref{fact}, if $B$ is a bounded set in $\e_1$, then for any strong solution $u(t),\ t\in[0,T]$ with initial data $\xi_u(0)\in B$, we have the uniform estimate
        \[\n\xi_u(t)\n_{\e_1}\leq Q(B,T,\n g\n_{L^2}),\quad \forall t\in[0,T]\]
        for some function $Q$ monotone increasing in $T$ and $\n g\n_{L^2}$.\\
        \textbf{Step 2.} \emph{Uniqueness and continuous dependence.} Suppose that $u_1$ and $u_2$ are two strong solutions of problem (\ref{eq1.1}) with initial data $\n\xi_{u_i}(0)\n_{\e_1}\leq r\ (i=1,2)$. According to \textbf{Step 1}, 
        \[\n\xi_{u_i}\n_{L^\infty(0,T;\e_1)}\leq C_{r,T},\quad i=1,2,\]
        for some constant $C_{r,T}$ dependent of $r$ and $T$. Let $w(t)=u_1(t)-u_2(t)$. Using the growth condition (\ref{1.2}), we have for $i=1,2,$
        \[\n \Delta f(u_i)\n_{L^2}=\n f'(u_i)\Delta u_i+f''(u_i)|\nabla u_i|^2\n_{L^2}\leq C(1+\n u_i\n_{H^2}^4)\n u_i\n_{H^2},\]
        i.e. $f(u_i)\in L^\infty(0,T;H^2)$. Let $w(t)=u_1(t)-u_2(t)$.
        Then $w(t)$ solves
        \[w_{tt}+\gamma w_t-\Delta w=f(u_2)-f(u_1)\in L^\infty(0,T;H^1_0),\quad \xi_w(0)=\xi_{u_1}(0)-\xi_{u_2}(0).\]
        Multiplying this equation by $-\Delta w_t$ and integrating over $x\in\Omega$, we have
        \begin{equation}\label{3-6-11}
        	\frac{1}{2}\frac{d}{dt}\n\xi_w(t)\n_{\e_1}^2+\gamma\n\nabla w_t\n_{L^2}^2=(f(u_2)-f(u_1),-\Delta w_t).
        \end{equation}
       Using the Sobolev embedding $H^2\hookrightarrow C$, we estimate the term at the right-hand side of (\ref{3-6-11}) as follows:
        \begin{align*}
        	&|(f(u_2)-f(u_1),-\Delta w_t)|=\left|\int_{0}^{1}(f'(u_2-\theta w)w,-\Delta w_t)d\theta\right|\\
        	&\leq\left|\int_{0}^{1}(f''(u_2-\theta w)(\nabla u_2-\theta\nabla w)w,\nabla w_t)d\theta\right|
        	+\left|\int_{0}^{1}(f'(u_2-\theta w)\nabla w,\nabla w_t)d\theta\right|\\
        	&\leq C((1+|u_1|^3+|u_2|^3)(|\nabla u_1|+|\nabla u_2|)|w|,|\nabla w_t|)
        	+C((1+|u_1|^4+|u_2|^4)|\nabla w|,|\nabla w_t|)\\
        	&\leq C[1+(\n u_1\n^3_{H^2}+\n u_2\n^3_{H^2})\n(|\nabla u_1|+|\nabla u_2|)\n_{L^2}+\n u_1\n^4_{H^2}+\n u_2\n^4_{H^2}]\cdot
        	\n \Delta w\n_{L^2}\n \nabla w_t\n_{L^2}\\
        	&\leq C_{r,T}\n\xi_w\n_{\e_1}^2.
        \end{align*}
        Inserting this estimate into equality (\ref{3-6-11}) and applying the Gronwall inequality, we get
        \[\n\xi_w(t)\n_{\e_1}^2\leq e^{C_{r,T}T}\n\xi_w(0)\n_{\e_1}^2,\quad \forall t\in[0,T].\]
        Hence the estimate (\ref{3-6-0}) holds and the proof is completed.
	\end{proof}
	Theorem (\ref{well}) implies that the Shatah-Struwe solution semigroup $S(t)$, when  restricted on $\e_1$, is a locally Lipschitz continuous semigroup on $\e_1$.

\section{Dissipation}
	In this section, we are intended to obtain the dissipativity of $(\e_1,S(t))$. Since the $\e-$norm of $\xi_u$ is already dissipative, due to the first inequality in estimate (\ref{3-6-5}), we just need to verify that the $\e-$norm of $\xi_v$ is also dissipative. The main difficulty is to deal with the term $(f'(u)v,v_t)$. In \cite{Liu1}, the authors concern the point dissipativity instead of the uniform dissipativity in the sub-quintic case. However, we will show that the uniform dissipativity can actually be obtained even in the quintic case, by taking full advantage of the fact that the dynamical system $(\e,S(t))$ possesses a compact global attractor.\\
	\\
	\emph{Sub-quintic case.} We firstly consider the sub-quintic case where the nonlinearity $f$ satisfies
	\begin{equation}\label{4.1}
		|f''(s)|\leq C(1+|u|^{3-\kappa})
	\end{equation}
	for some $0<\kappa\leq3$.
	\begin{thm}\label{thm4.1}
		Let $f$ satisfies conditions (\ref{4.1}), (\ref{1.3}) and (\ref{1.4}), then the dynamical system $(\e_1,S(t))$ is dissipative.
	\end{thm}
	\begin{proof}
		Let $B$ be an arbitrary bounded set in $\e_1$. Due to Lemma \ref{lemma3.3} and Remark \ref{remark3.4}, the system $(\e,S(t))$ has a bounded absorbing set $\mathcal{B}_0$ and then there exists a time moment $t_1=t_1(B)>0$, such that $S(t)B\subset\mathcal{B}_0,\ \forall t\geq t_1.$ Moreover,
		\begin{equation}\label{4-1-1}
			\n\xi_u\n_{L^\infty(t_1,\infty;\e)}+\n u\n_{L^4(t,t+1;L^{12})}\leq C_{\mathcal{B}_0},\quad \forall t\geq t_1.
		\end{equation}
		Again, let $v(t)=u_t(t)$, then by \textbf{Step 1} of Theorem \ref{well},
		\begin{equation}\label{4-1-2}
			\n\xi_v(t_1)\n_{\e}\leq C(\n\xi_u(t_1)\n_{\e_1})\leq C_B,
		\end{equation}
		 and we just need to estimate the $\e-$norm for $\xi_v(t)$ for time $t\geq t_1.$ To this end, multiplying Eq. (\ref{eq3.1}) by $v_t+\alpha v$ with $\alpha>0$ to be determined later, after integrating over $x\in\Omega$, we get
		 \begin{equation}\label{4-1-3}
		 	\frac{d}{dt}\Phi(v)+\alpha\Phi(v)+\alpha\n v_t\n_{L^2}^2+\Gamma=-2(f'(u)v,v_t)-2\alpha(f'(u)v,v),
		 \end{equation}
		where
		\[\Phi(v):=\n\xi_v\n_\e^2+2\alpha(v_t,v),\]
		\[\Gamma=(2\gamma-4\alpha)\n v_t\n_{L^2}^2+\alpha\n\nabla v\n_{L^2}^2+(2\alpha\gamma-2\alpha^2)(v_t,v).\]
		Choosing $\alpha=\alpha(\gamma,\lambda_1)>0$ small enough, we have $\Gamma>0$ and
		\begin{equation}\label{4-1-4}
			\frac{1}{2}\n\xi_v\n_\e^2\leq\Phi(v)\leq\frac{3}{2}\n\xi_v\n_\e^2,
		\end{equation}
		From (\ref{4.1}), we can see that $\forall\varepsilon>0,\ \exists\ C_\varepsilon>0$, such that
		\begin{equation}\label{4-1-5}
			|f'(u)|\leq \varepsilon |u|^4+C_\varepsilon.
		\end{equation}
		Then the first term on the right-hand side of (\ref{4-1-3}) obeys the estimate
		\begin{align*}
			-2(f'(u)v,v_t)&\leq \varepsilon(|u|^4|v|,|v_t|)+C_\varepsilon(|v|,|v_t|)\\
			&\leq \varepsilon\n u\n_{L^{12}}^4\n v\n_{L^6}\n v_t\n_{L^2}+C_\varepsilon\n u_t\n_{L^2}\n v_t\n_{L^2}\\
			&\leq \varepsilon\n u\n_{L^{12}}^4\n\xi_v\n_{\e}+\alpha\n v_t\n_{L^2}^2+C_{\varepsilon,\alpha}\n u_t\n_{L^2}^2
		\end{align*}
		for all $\varepsilon>0$ small enough. Owing to (\ref{1.4}), the last term in (\ref{4-1-3}) obeys the estimate
		\[-2\alpha(f'(u)v,v)\leq 2\alpha K\n v\n_{L^2}^2=2\alpha K\n u_t\n_{L^2}^2.\]
		Inserting these equalities into (\ref{4-1-3}), we have 
		\[\frac{d}{dt}\Phi(v)+\alpha\Phi(v)\leq \varepsilon\n u\n_{L^{12}}^4\Phi(v)+C_{\varepsilon,\alpha}\n u_t\n_{L^2}^2.\]
		Due to (\ref{4-1-1}),
		\[\int_{s}^{t}\varepsilon\n u(\tau)\n_{L^{12}}^4d\tau\leq \varepsilon C_{\mathcal{B}_0}(t-s+1),\quad \forall t\geq s\geq t_1,\]
		\[\n u_t\n_{L^2}^2\leq \n\xi_u\n_{L^\infty(t_1,\infty;\e)}^2\leq C_{\mathcal{B}_0},\quad \forall t\geq t_1.\]
		Fixing now $\varepsilon<\frac{\alpha}{2C_{\mathcal{B}_0}}$ and applying Lemma \ref{Gronwall}, we get 
		\[\Phi(v(t))\leq \Phi(v(t_1))e^{-\frac{\alpha}{2}(t-t_1)}+C_{\mathcal{B}_0},\quad \forall t\geq t_1.\]
		Combing this estimate with (\ref{3-6-5}), (\ref{4-1-1}), (\ref{4-1-2}) and (\ref{4-1-4}), we infer that
		\[\n\xi_u(t)\n_{\e_1}\leq C_Be^{-\frac{\alpha}{4}(t-t_1)}+C{(\mathcal{B}_0,\n g\n_{L^2})},
		\quad \forall t\geq t_1.\]
		This implies that the dynamical system $(\e_1,S(t))$ is dissipative, and $\mathbb{B}:=\{\xi\in\e_1:\n\xi\n_{\e_1}\leq 1+C{(\mathcal{B}_0,\n g\n_{L^2})}\}$ is a bounded absorbing set.
	\end{proof}
~\\
	\emph{Quintic case.} Our main aim is to obtain the dissipativity in the quintic case where the nonlinearity $f$ satisfies the growth condition (\ref{1.2}). The main difficulty comes from that the estimates (\ref{3-3}) and (\ref{4-1-5}) are both invalid here, hence, we cannot construct the suitable Gronwall inequality at least in a direct way. To solve this problem, inspired by \cite{Zelik1}, we first establish the uniform boundedness of $L^4(L^{12})$-norm of the solutions and then prove that the set of trajectories restricted on the interval $[0,1]$ constitutes a compact subset on $L^4(0,1;L^{12})$. Using this compactness, we can make a decomposition $u(t)=\un(t)+\um(t)$ where the $L^4(L^{12})-$norm of $\un$ is small and the other function $\um$ is smooth.\\
	\indent For any $\delta>0$, denote the $\delta-$neighbourhood of $\mathcal{A}_0$ in $\e$ as
	\[\mathcal{B}_\delta:=\{\xi\in\e:dist_\e(\xi,\mathcal{A}_0)\leq\delta\},\]
	where $\mathcal{A}_0$ is the compact global attractor of $(\e,S(t))$, $dist_\e(\cdot,\cdot)$ denotes the Hausdorff semi-distance in $\e$:
	\[dist_\e(A,B):=\sup_{x\in A}\inf_{y\in B}\n x-y\n_{\e},\quad A,B\subset\e.\]
	Obviously, $\mathcal{B}_{\delta}$ is a bounded absorbing set of $(\e,S(t))$ for any $\delta>0$ and we have the following basic result.
	\begin{lemma}\label{lemma4.2}
		For some $\delta_0>0$ small enough, there exist a time moment $T=T(\mathcal{A}_0)>0$ and a constant $C_0=C_0(\mathcal{A}_0)>0$ such that
		\[\n u\n_{L^4(0,T;L^{12})}\leq C_0\]
		for any Shatah-Struwe solution $u(t)$ with initial data $\xi_u(0)\in\mathcal{B}_{\delta_0}$.
	\end{lemma}
	\begin{proof}
		The proof is almost the same as the one of Proposition \ref{fact}. We again split the solution $u=v+w$ where $v,w$ solve the following problems respectively
		\[v_{tt}+\gamma v_t-\Delta v=g,\quad \xi_v(0)=\xi_u(0);\]
		\[w_{tt}+\gamma w_t-\Delta w=-f(v+w),\quad \xi_w(0)=0.\]
		For variable $v$, since $\mathcal{A}_0$ is compact, due to Lemma \ref{coro2.4}, for any $\varepsilon>0$, there exists $T=T(\varepsilon,\mathcal{A}_0)>0$, such that
		\begin{equation}\label{4-2-1}
			\n \xi_v(t)\n_\e\leq C,\quad \n v\n_{L^4(0,t;L^{12})}\leq\frac{\varepsilon}{2},\quad \forall t\leq T,\ \forall\xi_v(0)\in\mathcal{A}_0.
		\end{equation}
		\indent Let $\delta\leq \frac{\varepsilon}{2C_T}$, where $C_T$ denotes the constant in the Strichartz estimate (\ref{Stri}), then for any $\xi_v(0)\in\mathcal{B}_\delta$, there exists $\xi_{\tilde{v}}(0)\in \mathcal{A}_0$, such that
		\[\n\xi_v(0)-\xi_{\tilde{v}}(0)\n_{\e}\leq\frac{\varepsilon}{2C_T}.\]
		The function $v-\tilde{v}$ solves
		\[(v-\tilde{v})_{tt}+\gamma(v-\tilde{v})_t-\Delta(v-\tilde{v})=0.\]
		By Strichartz estimate (\ref{Stri}),
		\[\n v-\tilde{v}\n_{L^4(0,t;L^{12})}\leq C_T\n\xi_v(0)-\xi_{\tilde{v}}(0)\n_{\e}\leq \frac{\varepsilon}{2},\quad \forall t\leq T.\]
		Combining with (\ref{4-2-1}), we have
		\[\n \xi_v(t)\n_\e\leq C,\quad \n v\n_{L^4(0,t;L^{12})}\leq\varepsilon,\quad \forall t\leq T,\ \forall\xi_v(0)\in\mathcal{B}_\delta.\]
		Arguing exactly as in the rest proof of Proposition \ref{fact}, we get that, if $\varepsilon$ and $T$ are chosen small enough, then
		\[\n u\n_{L^4(0,T;L^{12})}\leq C_{\varepsilon,T},\quad\forall\xi_u(0)\in\mathcal{B}_\delta.\]
		So, we can fix $\varepsilon_0$ and $T_0=T_0(\varepsilon_0,\mathcal{A}_0)$ being small to satisfy the above inequality, and take $\delta_0=\frac{\varepsilon_0}{2C_{T_0}}.$ The proof is complete.
	\end{proof}
	Now, let $B$ be an arbitrary bounded set in $\e_1$. Then there exists $t_1=t_1(B)>0$, such that
	\[S(t)B\subset\mathcal{B}_{\delta_0},\quad \forall t\geq t_1.\]
	Define 
	\[\bw:=\left[\bigcup_{s\geq t_1}S(s)B\right]_\e,\]
	where $[A]_\e$ denotes the closure of the set $A$ in $\e.$ Obviously, $\bw\subset\mathcal{B}_{\delta_0}$ and we have the following results.
	\begin{lemma}\label{lemma4.3}
		$\bw$ is a positively invariant set of $S(t)$, and a compact set in $\e.$
	\end{lemma}
	\begin{proof}
		Since $S(t)$ is a continuous semigroup on $\e$, $\bw$ is obviously positively invariant by definition. So, it suffices to prove that for any sequence $\{y_n\}_{n=1}^\infty\subset\bigcup_{s\geq t_1}S(s)B$, there exists a subsequence $\{y_{n_k}\}_{k=1}^\infty$ which converges in $\e.$\\
		\indent Case 1. $\exists t_*>t_1$, such that 
		\[\{y_n\}_{n=1}^\infty\subset\bigcup_{t_1\leq s\leq t_*}S(s)B.\]
		Due to the \textbf{Step 1} of Theorem \ref{well}, we have 
		\[\sup_{n\geq1}\n y_n\n_{\e_1}\leq\sup_{t_1\leq s\leq t_*}\sup_{\xi_0\in B}\n S(s)\xi_0\n_{\e_1}\leq C_{{B},t_*},\]
		i.e. $\{y_n\}_{n=1}^\infty$ is a bounded set in $\e_1$ and thus there is a convergent subsequence in $\e.$\\
		\indent Case 2. $\exists\ t_n\rightarrow\infty$ and $x_n\in B$, such that $y_n=S(t_n)x_n$. Then,  $\exists\ t_{n_k}\rightarrow\infty$, such that
		\[y_{n_k}=S(t_{n_k})x_{x_k}\rightarrow y_0\in\mathcal{A}_0\quad in\ \e.\]
		\indent The proof is complete.
	\end{proof}
	\begin{lemma}\label{lemma4.4}
		Let\[\mathcal{K}:=\{u(t)|_{t\in\mathbb{R}^+}:u(t)\ is\ the\ Shatah-Struwe\ solution\ with\ initial\ data\ \xi_u(0)\in\bw\}.\]
		Then $\mathcal{K}|_{t\in[0,1]}$ is a compact subset of $L^4(0,1;L^{12})$, i.e.
		\[\mathcal{K}|_{t\in[0,1]}\subset\subset L^4(0,1;L^{12}).\]
		Moreover, $\forall u\in\mathcal{K},$
		\begin{equation}\label{4-4-1}
			\sup_{t\geq0}\n u\n_{L^4(t,t+1;L^{12})}\leq C_1,
		\end{equation}
		where the constant $C_1$ depends only on $\mathcal{A}_0$, but independs on $\bw.$ 
	\end{lemma}
    \begin{proof}
    	Firstly, we show that $\mathcal{K}$ is positively invariant under the translation: $T_h\mathcal{K}\subset\mathcal{K},\ \forall h\geq0$, where $(T_hu)(t):=u(t+h)$.\\
    	\indent Indeed, $\forall v\in T_h\mathcal{K}$, there exists $u\in\mathcal{K}$ such that 
    	\[v(t)=(T_hu)(t)=u(t+h).\]
    	Since $u$ is a Shatah-Struwe solution and $\bw$ is positively invariant under $S(t)$, we have $v$ is also a Shatah-Struwe solution and $\xi_v(0)=\xi_u(h)=S(h)\xi_u(0)\in \bw.$ This shows that $v\in\mathcal{K}.$\\
    	\indent Secondly, since $\bw\subset\mathcal{B}_{\delta_0}$, by Lemma \ref{lemma4.2}, there is $T=T(\mathcal{A}_0)>0$ such that 
    	\[\n u\n_{L^4(0,T;L^{12})}\leq C_0,\quad \forall u\in\mathcal{K}.\]
    	Then, since $T_h\mathcal{K}\subset\mathcal{K}$, we can deduce that
    	\begin{equation}\label{4-4-2}
    		\sup_{t\geq0}\n u\n_{L^4(t,t+1;L^{12})}\leq\frac{C_0}{\min\{T,1\}}=: C_1,\quad \forall u\in \mathcal{K},
    	\end{equation}
    	where the constant $C_1$ depends only on $\mathcal{A}_0$.\\
    	\indent Now, define a map $S:\bw\rightarrow L^4(0,1;L^{12})$, via $S:\xi_u(0)\longmapsto u|_{t\in[0,1]}$. We want to verify the continuity of $S$. 
    	Let $u_1(t),\ u_2(t)\in\mathcal{K}|_{t\in[0,1]}$ with initial data $
    	\xi_{u_1}(0),\ \xi_{u_2}(0)$ respectively and let $w(t)=u_1(t)-u_2(t)$. Then $w$ solves
    	\[w_{tt}+\gamma w-\Delta w=-f(u_1)+f(u_2).\]
    	Mulitiplying this equation by $w_t$ and integrating over $x\in\Omega$, 
    	\[\frac{1}{2}\frac{d}{dt}\n\xi_w(t)\n_{\e}^2+\gamma\n w_t\n_{L^2}^2=-(f(u_1)-f(u_2),w_t).\]
    	We estimate the last term as follows:
    	\begin{align*}
    		|(f(u_1)-f(u_2),w_t)|&\leq C((1+|u_1|^4+|u_2|^4)|w|,|w_t|)\\
    		&\leq C(1+\n u_1\n_{L^{12}}^4+\n u_2\n_{L^{12}}^4)\n w\n_{L^6}\n w_t\n_{L^2}\\
    		&\leq C(1+\n u_1\n_{L^{12}}^4+\n u_2\n_{L^{12}}^4)\n\xi_w\n_{\e}^2.
    	\end{align*}
    	Applying the Gonwall inequality and use the estimate (\ref{4-4-2}), we end up with
    	\[\n\xi_{u_1}(t)-\xi_{u_2}(t)\n_{\e}\leq Ce^{C_1t}\n\xi_{u_1}(0)-\xi_{u_2}(0)\n_{\e},\quad \forall t\in[0,1].\]
    	Then,
    	\begin{align*}
    		&\n f(u_1)-f(u_2)\n_{L^1(0,1;L^{2})}\\
    		&\qquad\leq C\int_{0}^{1}(1+\n u_1\n_{L^{12}}^4+\n u_2\n_{L^{12}}^4)\n u_1-u_2\n_{L^6}dt\\
    		&\qquad\leq C(1+||u_1||^4_{L^4(0,1;L^{12})}+||u_2||^4_{L^4(0,1;L^{12})})||\xi_{u_1}-\xi_{u_2}||_{L^\infty(0,1;\e)}\\
    		&\qquad \leq C\n\xi_{u_1}(0)-\xi_{u_2}(0)\n_{\e}.
    	\end{align*}
        Applying the Strichartz estimate (\ref{Stri}), we get
        \begin{align*}
        	\n u_1-u_2\n_{L^4(0,1;L^{12})}&\leq C\n\xi_{u_1}(0)-\xi_{u_2}(0)\n_\e+C\n f(u_1)-f(u_2)\n_{L^1(0,1;L^{2})}\\
        	&\leq C\n\xi_{u_1}(0)-\xi_{u_2}(0)\n_\e.
        \end{align*}
        Thus, the map $S:\e\rightarrow L^4(0,1;L^{12})$ is continuous on $\bw$. Since $\bw$ is compact in $\e$, $
        \mathcal{K}|_{t\in[0,1]}=S(\bw)$ is compact in $L^4(0,1;L^{12})$.
    \end{proof}

	\begin{lemma}\label{lemma4.5}
		For any $\varepsilon>0$ and any $u\in\mathcal{K}$, we can make a decomposition $u(t)=\un(t)+\um(t)$, such that
		\begin{equation}
			\sup_{t\geq0}\n \un \n_{L^4(t,t+1;L^{12})}\leq\varepsilon
		\end{equation}
		and 
		\[\n \um\n_{L^\infty(\mathbb{R}^+;H^2)}\leq C_\varepsilon,\]
		where the constant $C_\varepsilon$ depends on $\varepsilon$ and $C_1$ in (\ref{4-4-1}), but independs on $\bw$.
	\end{lemma}
	\begin{proof}
		Due to the compactness $\mathcal{K}|_{t\in[0,1]}\subset\subset L^4(0,1;L^{12})$, for any $\varepsilon>0$, there is a finite set $\{w^i\}_{i=1}^m\subset \mathcal{K}|_{t
			\in[0,1]}$ such that
		\[\mathcal{K}|_{t\in[0,1]}\subset\bigcup_{1\leq i\leq m}B_{L^4(0,1;L^{12})}(w^i,\frac{\varepsilon}{4}),\]
		where $B_X(x_0,r)$ denotes the $r-$ball centered on $x_0$ in the space $X.$ In addition, we have by (\ref{4-4-1}),
		\[\sup_{1\leq i\leq m}\n w^i\n_{L^4(0,1;L^{12})}\leq C_1.\]
		We can approximate $w^i$ by smoother functions, namely, there are $\{\tilde{w}^i\}_{i=1}^m$ such that 
		\[\n \tilde{w}^i-w^i\n_{L^4(0,1;L^{12})}\leq\frac{\varepsilon}{4}\quad and\quad \n \tilde{w}^i\n_{C(0,1;H^2)}\leq C_\varepsilon,\]
		where the constant $C_\varepsilon$ is dependent on $\varepsilon$ and $C_1$, but is independent of $w^i.$ Then,
		\[\mathcal{K}|_{t\in[0,1]}\subset\bigcup_{1\leq i\leq m}B_{L^4(0,1;L^{12})}(\tilde{w}^i,\frac{\varepsilon}{2}).\]
		\indent Now, for every $u\in\mathcal{K}$ and for any $n\in\mathbb{N}$, choose $\tilde{w}^{i_n}$ such that
		\[u|_{t\in[n,n+1]}=T_nu|_{t\in[0,1]}\in B_{L^4(0,1;L^{12})}(\tilde{w}^{i_n},\frac{\varepsilon}{2}).\]
		Define the function $\um(t)$ as 
		\[\um(t)=\tilde{w}^{i_n}(t-n),\ if\ t\in[n,n+1),\ \forall n\in\mathbb{N},\]
		and the function $\un(t)=u(t)-\um(t).$ Then, 
		\[\n \um\n_{L^\infty(\mathbb{R}^+;H^2)}\leq\sup_{n\in\mathbb{N}}\n \tilde{w}^{i_n}(t-n)\n_{C(n,n+1;H^2)}\leq C_\varepsilon,\]
		and for any $t\geq0,$ $[t,t+1]\subset[n,n+2)$ for some $n$,
		\begin{align*}
			\n \un\n_{L^4(t,t+1;L^{12})}&\leq\n u-\um\n_{L^4(n,n+1;L^{12})}+\n u-\um\n_{L^4(n+1,n+2;L^{12})}\\
			&\leq \n T_nu-\tilde{w}^{i_n}\n_{L^4(0,1;L^{12})}+\n T_{n+1}u-\tilde{w}^{i_{n+1}}\n_{L^4(0,1;L^{12})}\\
			&\leq \frac{\varepsilon}{2}+\frac{\varepsilon}{2}\leq\varepsilon.
		\end{align*}
		The Lemma is proved.
	\end{proof}
	Now, we are ready to prove the main result in this section.
	\begin{thm}\label{dissipation}
		Let $f$ satisfy assumptions (\ref{1.2})-(\ref{1.7}). Then, the dynamical system $(\e_1,S(t))$ is dissipative.
	\end{thm}
	\begin{proof}
		We will proceed analogously to the proof of Theorem \ref{thm4.1}, but will estimate the term $(f'(u)v,v_t)$ in a different way. Take any initial data $\xi_u(0)\in\bw$. Due to Lemma \ref{lemma4.5}, for any $\varepsilon>0$, we can split the solution into $u(t)=\un(t)+\um(t)$, such that  
		\begin{equation}\label{4-6-1}
			\sup_{t\geq0}\n \un \n_{L^4(t,t+1;L^{12})}\leq\varepsilon
		\end{equation}
		and 
		\[\n \um\n_{L^\infty(\mathbb{R}^+;H^2)}\leq C_\varepsilon,\]
		where the constant $C_\varepsilon$ depends only on $\varepsilon$ and $\mathcal{A}_0$. Using this decomposition, we have the following estimate:
		\begin{align*}
			|(f'(u)v,v_t)|&\leq |((f'(\un+\um)-f'(\um))v,v_t)|+|(f'(\um)v,v_t)|\\
			&\leq C((1+|\un|^3+|\um|^3)|\un|,|v||v_t|)+||f'(\um)||_{L^\infty}||v||_{L^2}||v_t||_{L^2}\\
			&\leq C(1+||\un||_{L^{12}}^3+||u||_{L^{12}}^3)||\un||_{L^{12}}\n v\n_{L^6}\n v_t\n_{L^2}\\
			&\quad +C(1+\n \um\n^4_{H^2})||u_t||_{L^2}||v_t||_{L^2}\\
			&\leq C(1+||\un||_{L^{12}}^3+||u||_{L^{12}}^3)||\un||_{L^{12}}\n \xi_v\n^2_{\e} +C_{\varepsilon,\alpha}||u_t||_{L^2}^2+\alpha\n v_t\n^2_{L^2}\\
			&\leq l_\varepsilon(t)\n \xi_v\n^2_{\e}+C_{\varepsilon,\alpha}||u_t||_{L^2}^2+\alpha\n v_t\n^2_{L^2},
		\end{align*}
	    where $l_\varepsilon(t)=C(1+||\un||_{L^{12}}^3+||u||_{L^{12}}^3)||\un||_{L^{12}}.$
	    Due to (\ref{4-4-1}) and (\ref{4-6-1}), we have for any $t\geq0,$
	    \begin{align*}
	    	\int_{t}^{t+1}l_\varepsilon(\tau)d\tau&\leq C\left(\int_{t}^{t+1}(1+||\un||_{L^{12}}^3+||u||_{L^{12}}^3)^\frac{4}{3}\right)^\frac{3}{4}\cdot\left(\int_{t}^{t+1}||\un||_{L^{12}}^4\right)^\frac{1}{4}\\
	    	&\leq C(1+||\un||_{L^4(t,t+1;L^{12})}^3+||u||_{L^4(t,t+1;L^{12})}^3)\cdot ||\un||_{L^4(t,t+1;L^{12})}\\
	    	&\leq C\varepsilon,
	    \end{align*}
        where the constant $C$ is independent of $\varepsilon$. This implies that 
        \[\int_{s}^{t}l_\varepsilon(\tau)d\tau\leq C\varepsilon(t-s+1),\quad\forall t\geq s\geq0.\]
        Since, $S(t)\bw\subset\bw\subset\mathcal{B}_{\delta_0},\ \forall t\geq0$, we have
        \[||u_t||_{L^2}^2\leq\n\xi_u(t)\n_{\e}^2\leq C_{\mathcal{B}_{\delta_0}}=C(\delta_0,\mathcal{A}_0),\quad\forall t\geq0.\]
        Fixing now $\varepsilon>0$ to be small enough and arguing as in the end of Theorem \ref{thm4.1}, we arrive at 
        \[\n\xi_u(t)\n_{\e_1}\leq C(\n\xi_u(0)\n_{\e_1})e^{-\frac{\alpha}{4}t}+C(\delta_0,\mathcal{A}_0,\n g\n_{L^2}),\quad \forall t\geq0,\ \xi_u(0)\in\bw.\]
        In particular, for any $\xi_u(0)\in S(t_1)B$, the above estimate holds. Then, there exists $t_2=t_2(B)>0$, such that
        \[\n S(t)\xi_0\n_{\e_1}\leq 1+C(\delta_0,\mathcal{A}_0,\n g\n_{L^2}),\quad\forall \xi_0\in B,\ \forall t\geq t_1+t_2.\]
        Therefore, the semigroup $(\e_1,S(t))$ is dissipative and the set $\tilde{\mathbb{B}}:=\{\xi\in\e_1:\n\xi\n_{\e_1}\leq1+ C(\delta_0,\mathcal{A}_0,\n g\n_{L^2})\}$ is a bounded absorbing set.
	\end{proof}

\section{Quasi stability}
	In this section, we want to show that the dynamical system $(\e_1,S(t))$ is quasi-stable on every positively invariant bounded set (see Definition \ref{defquasi}).
    \begin{lemma}\label{lemma5.1}
    	Under the assumption (\ref{1.2}), the operator
    	\begin{align*}
    		T_f:H^2\cap H^1_0&\longrightarrow H^1_0\\
    		&u\longmapsto f(u)
    	\end{align*}
    	is compact.
    \end{lemma}
	\begin{proof}
		For any $u,v\in H^2(\Omega)\cap H^1_0(\Omega)$,
		\begin{align*}
			\n f(u)-f(v)\n_{H^1_0(\Omega)}&=\n f'(u)\nabla u-f'(v)\nabla v\n_{L^2}\\
			&\leq \n f'(u)(\nabla u-\nabla v)\n_{L^2}+\n (f'(u)-f'(v))\nabla v\n_{L^2}\\
			&\leq C(1+\n u\n_{H^2}^4)\n\nabla(u-v)\n_{L^2}\\
			&\quad +C(1+\n u\n_{H^2}^3+\n v\n_{H^2}^3)\n\nabla v\n_{L^3}\n u-v\n_{L^6}\\
			&\leq C(1+\n u\n_{H^2}^4+\n v\n_{H^2}^4)\n u-v\n_{H^1_0}.
		\end{align*}
		Since the embedding $H^2\cap H^1_0\hookrightarrow H^1_0$ is compact, we assert that $T_f$ is a compact operator.
	\end{proof}
	
	\begin{lemma}\label{lemma5.2}
		Let the assumptions of Theorem \ref{dissipation} hold. Then there exists $T_0>0$ and a constant $c$ independent of $T$, such that for any pair $u(t)$ and $v(t)$ of strong solution to problem (\ref{eq1.1}), we have the following relation
		\begin{equation}
			TE_z(T)+\int_{0}^{T}E_z(t)dt\leq c\left\{\int_{0}^{T}\n \nabla z_t\n_{L^2}^2dt+\left|\int_{0}^{T}(\nabla z_t,\nabla z)dt\right|+\Psi_T(u,v)\right\}
		\end{equation}
		for any $T\geq T_0$, where $z(t)=u(t)-v(t)$ and we use the following notations:
		\[E_z(t)=\frac{1}{2}\n\xi_z\n_{\e_1}^2=\frac{1}{2}(\n\nabla z_t\n_{L^2}^2+\n\Delta z\n_{L^2}^2),\]
		\begin{align*}
			\Psi_T(u,v)&=\left|\int_{0}^{T}(\fot,-\Delta z)dt\right|+\left|\int_{0}^{T}(\fot,-\Delta z_t)dt\right|\\
			&+\left|\int_{0}^{T}\int_{t}^{T}(\fot,-\Delta z_t) d\tau dt\right|
		\end{align*}
		with $G_{u,v}(t)=f(u(t))-f(v(t))$.
	\end{lemma}
	\begin{proof}
		The variable $z$ solves the equation 
		\begin{eqnarray}\label{z}
			z_{tt}+\gamma z_t-\Delta z=f(v(t))-f(u(t)),\quad \xi_z(0)=\xi_u(0)-\xi_v(0).
		\end{eqnarray}
	    By Lemma \ref{lemma5.1}, $f(v(t))-f(u(t))\in L^\infty(0,T;H^1_0)$. 
		Multiplying (\ref{z}) by $-\Delta z_t$ and integrating over $x\in\Omega$,
		\begin{eqnarray}\label{one}
			\frac{d}{dt}E_z(t)+\gamma\n\nabla z_t\n_{L^2}^2+(\fot,-\Delta z_t)=0.
		\end{eqnarray}
		Integrating (\ref{one}) on $[t,T]$,
		\[E_z(T)+\gamma\int_{t}^{T}\n \nabla z_t\n_{L^2}^2d\tau=E_z(t)-\int_{t}^{T}(\fot,-\Delta z_t)d\tau.\]
		Integrating the above equality w.r.t. $t$ on $[0,T]$,
		\begin{equation}\label{two}
			TE_z(T)+\gamma\int_{0}^{T}\int_{t}^{T}\n \nabla z_t\n_{L^2}^2d\tau dt=\int_{0}^{T}E_z(t)dt-\int_{0}^{T}\int_{t}^{T}(\fot,-\Delta z_t) d\tau dt.
		\end{equation}
		Multiplying (\ref{z}) by $-\Delta z$ and integrating over $x\in\Omega$,
		\[\frac{d}{dt}(\nabla z_t,\nabla z)-\n \nabla z_t\n_{L^2}^2+\gamma(\nabla z_t,\nabla z)+\n\Delta z\n_{L^2}^2+(\fot,-\Delta z)=0,\]
		Integrating this equation on $[0,T],$
		\begin{equation}\label{three}
			\begin{split}
				\int_{0}^{T}E_z(t)dt&=\frac{1}{2}(\nabla z_t(0),\nabla z(0))-\frac{1}{2}(\nabla z_t(T),\nabla z(T))
				+\int_{0}^{T}\n \nabla z_t\n_{L^2}^2dt\\
				&\quad -\frac{\gamma}{2}\int_{0}^{T}(\nabla z_t,\nabla z)dt-\frac{1}{2}\int_{0}^{T}(\fot,-\Delta z) dt\\
				&\leq c(E_z(0)+E_z(T))+\int_{0}^{T}\n \nabla z_t\n_{L^2}^2dt\\
				&\quad +\frac{\gamma}{2}\left|\int_{0}^{T}(\nabla z_t,\nabla z)dt\right|+\frac{1}{2}\left|\int_{0}^{T}(\fot,-\Delta z) dt\right|,
			\end{split}
		\end{equation}
		where the constant $c$ depends only on $\Omega$. Integrating (\ref{one}) on $[0,T]$,
		\begin{equation}\label{energy}
			E_z(0)=E_z(T)+\gamma\int_{0}^{T}\n \nabla z_t\n_{L^2}^2dt+\int_{0}^{T}(\fot,-\Delta z_t)dt.
		\end{equation}
		Inserting this quation into (\ref{three}), we have
		\begin{equation}\label{four}
			\begin{split}
				\int_{0}^{T}E_z(t)dt
				&\leq c\left(E_z(T)+\int_{0}^{T}\n \nabla z_t\n_{L^2}^2dt+\left|\int_{0}^{T}(\nabla z_t,\nabla z)dt\right|\right)\\
				&\quad +c\left(\left|\int_{0}^{T}(\fot,-\Delta z) dt\right|+\left|\int_{0}^{T}(\fot,-\Delta z_t)dt\right|\right).
			\end{split}
		\end{equation}
		Combing (\ref{two}) and (\ref{four}), we can choose $T_0\geq 2c$ such that
		\begin{align*}
			TE_z(T)+\int_{0}^{T}E_z(t)dt&\leq c\left(\int_{0}^{T}\n \nabla z_t\n_{L^2}^2dt+\left|\int_{0}^{T}(\nabla z_t,\nabla z)dt\right|+\Psi_T(u,v)\right),
		\end{align*}
		for every $T\geq T_0$, where
		\begin{align*}
			\Psi_T(u,v)&=\left|\int_{0}^{T}(\fot,-\Delta z) dt\right|+\left|\int_{0}^{T}(\fot,-\Delta z_t)dt\right|\\
			&\quad +\left|\int_{0}^{T}\int_{t}^{T}(\fot,-\Delta z_t) d\tau dt\right|.
		\end{align*}
	\end{proof}
	The main result in this section is the following theorem.
	\begin{thm}\label{quasi}
		Let the assumptions of Theorem \ref{dissipation} hold. Then the dynamical system $(\e_1,S(t))$ is quasi-stable on every positively invariant bounded set.
	\end{thm}
	\begin{proof}
		Let $B\subset\e_1$ be an arbitrary positively invariant bounded set. For every $\xi_1,\ \xi_2\in B$, let $\xi_u(t):=S(t)\xi_1,\ \xi_v(t):=S(t)\xi_2$ and $z(t)=u(t)-v(t).$ Since $S(t)B\subset B,\ \forall t\geq0$, we have 
		\begin{equation}\label{5-3-1}
			\n\xi_u\n_{L^\infty(\mathbb{R^+};\e_1)}+\n\xi_v\n_{L^\infty(\mathbb{R^+};\e_1)}\leq C_B.
		\end{equation}
	    Due to Lemma \ref{lemma5.2}), for fixed $T>T_0,$ we have
	    \begin{equation}\label{5-3-2}
	    	TE_z(T)+\int_{0}^{T}E_z(t)dt\leq c\left\{\int_{0}^{T}\n \nabla z_t\n_{L^2}^2dt+\left|\int_{0}^{T}(\nabla z_t,\nabla z)dt\right|+\Psi_T(u,v)\right\},
	    \end{equation}
	    where the notations $E_z(t)$ and $\Psi_T(u,v)$ are the same as in Lemma \ref{lemma5.2}.\\
		\indent We are going to estimate the right-hand side of (\ref{5-3-2}) term by term. Firstly, from the equality (\ref{energy}), 
		\begin{equation}\label{5-3-3}
			\gamma\int_{0}^{T}\n \nabla z_t\n_{L^2}^2dt=E_z(0)-E_z(T)+\int_{0}^{T}(\fot,-\Delta z_t)dt.
		\end{equation}
	    For the second term in the right hand side of inequality (\ref{5-3-2}),
	    \begin{equation}\label{5-3-4}
	    	\begin{split}
	    		\left|\int_{0}^{T}(\nabla z_t,\nabla z)dt\right|&=\left|\frac{1}{2}\int_{0}^{T}\frac{d}{dt}\n\nabla z\n_{L^2}^2dt\right|\\
	    		&\leq \frac{1}{2}(\n\nabla z(0)\n_{L^2}^2+\n\nabla z(T)\n_{L^2}^2)\\
	    		&\leq\sup_{0\leq t\leq T}\n z\n_{H^1_0}^2.
	    	\end{split}
	    \end{equation}
	    According to Lemma \ref{lemma5.1} and the estimate (\ref{5-3-1})
	    \[\n \fot\n_{L^\infty(0,T;H^1_0)}=\n f(u)-f(v)\n_{L^\infty(0,T;H^1_0)}\leq C_B\cdot\n u-v\n_{L^\infty(0,T;H^1_0)}=C_B\cdot\sup_{0\leq t\leq T}\n z\n_{H^1_0},\]
	    Then, for any $\varepsilon>0$,
	    \begin{equation}\label{5-3-5}
	    	\begin{split}
	    		\int_{0}^{T}\left|(\fot,-\Delta z_t)\right|dt&\leq\left|\int_{0}^{T}\n \fot\n_{H^1_0}\n-\Delta z_t\n_{H^{-1}}dt\right|\\
	    		&\leq C\left|\int_{0}^{T}\n \fot\n_{H^1_0}\n z_t\n_{H^1_0}dt\right|\\
	    		&\leq C\left|\int_{0}^{T}C_\varepsilon\n \fot\n_{H^1_0}^2+\frac{\varepsilon}{C}\n z_t\n_{H^1_0}^2dt\right|\\
	    		&\leq C_{B,\varepsilon}T\cdot\sup_{0\leq t\leq T}\n z\n_{H^1_0}^2+\varepsilon\int_{0}^{T}E_z(t)dt.
	    	\end{split}
	    \end{equation}
	    and 
	    \begin{equation}\label{5-3-6}
	    	\begin{split}
	    		\left|\int_{0}^{T}(\fot,-\Delta z)dt\right|&\leq\left|\int_{0}^{T}\n \fot\n_{H^1_0}\n-\Delta z\n_{H^{-1}}dt\right|\\
	    		&\leq C\left|\int_{0}^{T}\n \fot\n_{H^1_0}\n z\n_{H^1_0}dt\right|\\
	    		&\leq C_{B}T\cdot\sup_{0\leq t\leq T}\n z\n_{H^1_0}^2.
	    	\end{split}
	    \end{equation}
	    Plugging (\ref{5-3-3}), (\ref{5-3-4}), (\ref{5-3-5}), (\ref{5-3-6}) into (\ref{5-3-2}), we have
	    \begin{align*}
	    	TE_z(T)+\int_{0}^{T}E_z(t)dt&\leq \frac{c}{\gamma}(E_z(0)-E_z(T))\\
	    	&\quad +c\cdot\sup_{0\leq t\leq T}\n z\n_{H^1_0}^2+c\left|\int_{0}^{T}(\fot,-\Delta z)dt\right|\\
	    	&\quad +c(\frac{1}{\gamma}+1+T)\int_{0}^{T}\left|(\fot,-\Delta z_t)\right|dt\\
	    	&\leq \frac{c}{\gamma}(E_z(0)-E_z(T))+C_{B,\varepsilon,T}\cdot\sup_{0\leq t\leq T}\n z\n_{H^1_0}^2\\
	    	&\quad +C_T\varepsilon\int_{0}^{T}E_z(t)dt.
	    \end{align*}
	    Choosing $\varepsilon\leq \frac{1}{2C_T}$, we end up with
	    \[	TE_z(T)\leq \frac{c}{\gamma}(E_z(0)-E_z(T))+C_{B,T}\cdot\sup_{0\leq t\leq T}\n z\n_{H^1_0}^2.\]
	    The above inequality implies that
	    \begin{equation}\label{5-3-7}
	    	E_z(T)\leq \frac{c}{\gamma T+c} E_z(0)+C_{B,T}\cdot\sup_{0\leq t\leq T}\n z\n_{H^1_0}^2.
	    \end{equation}
	    \indent Now, define a Banach space $Z$ as
	    \[Z=\left\{u(t)\in L^2(0,T;\hh):\n u\n_Z^2:=\int_{0}^{T}\n\nabla u_t\n_{L^2}^2+\n\Delta u\n_{L^2}^2dt<\infty\right\}.\]
	    Define the mapping
	    \[K:B\rightarrow Z,\ K\xi_u(0)=u(t),\ t\in[0,T],\quad\forall \xi_u(0)\in B.\]
	    Then, arguing as the proof in the \textbf{Step 2} of Theorem \ref{well}, we have for any $\xi_u(0),\xi_v(0)\in B,$
	    \begin{align*}
	    	\n K\xi_u(0)-K\xi_v(0)\n_Z^2&=\int_{0}^{T}\n\nabla w_t\n_{L^2}^2+\n\Delta w\n_{L^2}^2dt\\
	    	&\leq T\cdot e^{C_BT}(\n\nabla w_t(0)\n_{L^2}^2+\n\Delta w(0)\n_{L^2}^2)\\
	    	&\leq C'_{B,T}\n\xi_u(0)-\xi_v(0)\n_{\e_1}^2,
	    \end{align*}
	    where $w(t)=u(t)-v(t), t\in[0,T].$ That is to say that $K$ is globally Lipschitz on $B.$ Define the functional
	    \[n_Z:Z\rightarrow\mathbb{R},\ n_Z(u)=\sqrt{C_{B,T}}\cdot\n z\n_{C(0,T;H^1_0)},\]
	    where the constant $\sqrt{C_{B,T}}$ is the same as the one in (\ref{5-3-7}). By Lemma \ref{Aubin}, the space $Z$ is compactly embedded into $C(0,T;H^1_0)$, which implies that the seminorm $n_Z$ is compact on $Z$.\\
	    \indent It then follows from (\ref{5-3-7}) that
	    \[\n S(T)\xi_1-S(T)\xi_2\n_{\e_1}\leq\eta_T\n\xi_1-\xi_2\n_{\e_1}+n_Z(K\xi_1-K\xi_2),\]
	    where $\eta_T=\sqrt{{c}/(\gamma T+c)}\in(0,1).$
	    It then follows that the system $(\e_1,S(t))$ is quasi-stable on $B$ at time $T$.
	\end{proof}

\section{Global attractor}

    In this section, we are ready to verify the existence of the global attractor of $(\e_1,S(t))$. The finite fractal dimension of the global attractor is also obtained. For convenience of the ready, we recall the difinition of the global attractor and some criterion for its existence.
    \begin{Def}\cite{Hale,Temam,Babin}
    	Let $\{S(t)\}_{t\geq0}$ be a $C^0-$semigroup acting on a metric space $(X,d)$. A subset $\mathcal{A}\subset X$ is called a global attractor of $(X,S(t))$ if\\
    	(i) $\mathcal{A}$ is compact in $X$;\\
    	(ii) $\mathcal{A}$ is invariant, i.e. $S(t)\mathcal{A}=\mathcal{A},\ \forall t\geq0$;\\
    	(iii) $\mathcal{A}$ attracts all bounded sets in $X$, i.e. for any bounded set $B\subset X$, 
    	\[dist_X(S(t)B,\mathcal{A}):=\sup_{x\in B} \inf_{y\in \mathcal{A}}d(x,y)\rightarrow 0,\ as\ t\rightarrow\infty.\]
    \end{Def}
    One of a criterion for the existence of the global attractor is as follows.
    \begin{lemma}\cite{Hale}\label{hale}
    	Let $(X,S(t))$ be a $C^0-$semigroup in a Banach space $X$. Then $(X,S(t))$ possesses a global attractor if and only if\\
    	(i) $(X,S(t))$ is dissipative;\\
    	(ii) $(X,S(t))$ is asymptotically smooth.
    \end{lemma}
    Now, we are in a position to state our main result.
    \begin{thm}\label{main}
    	Let the assumptions of Theorem \ref{quasi} hold. Then the dynamical system $(\e_1,S(t))$ possesses a global attractor $\mathcal{A}_1$. Moreover, the attractor $\mathcal{A}_1$ has a finite fractal dimension $dim_f^{\e_1}(\mathcal{A}_1)$ which satisfies the estimate (\ref{dim}).
    \end{thm}
    \begin{proof}
    	Due to Theorem \ref{well} and Theorem \ref{dissipation}, the dynamical system $(\e_1,S(t))$ is constinuous and dissipative. From Proposition \ref{prop1} and Theorem \ref{quasi}, $(\e_1,S(t))$ is asymptotically smooth. Then we conclude from Lemma \ref{hale} that $(\e_1,S(t))$ possesses a global attractor $\mathcal{A}_1$.\\
    	\indent Moreover, from Theorem \ref{quasi}, the system is quasi-stable on the attractor $\mathcal{A}_1$ at some time $t_*$. Then by Proposition \ref{prop2}, $\mathcal{A}_1$ has a finite fractal dimension $dim_f^{\e_1}(\mathcal{A}_1)$ and the bound of $dim_f^{\e_1}(\mathcal{A}_1)$ is given in estimate (\ref{dim}).
    \end{proof}
	\begin{remark}
		Since the global attractor $\mathcal{A}_0$ of $(\e,S(t))$ is a bounded set in $\e_1$ (Lemma \ref{lemma3.3}), we can deduce that $\mathcal{A}_0=\mathcal{A}_1.$ Then from Theorem \ref{main}, we can obtain the finite fractal dimension of $\mathcal{A}_0$ in $\e$:
		\[dim_f^\e(\mathcal{A}_0)=dim_f^{\e}(\mathcal{A}_1)\leq dim_f^{\e_1}(\mathcal{A}_1)<\infty.\]
	\end{remark}

\section*{Acknowledgements}
	We would like to express our sincere thanks to the everyone valuable comments and suggestions which play an important role of our original manuscript. This work was supported by the National Science Foundation of China Grant (11731005) .
	
\section*{Data Availability Statement}	
	Data sharing is not applicable to this article as no new data were created or analyzed in this study.

\end{document}